\documentclass{amsart}
\usepackage[utf8]{inputenc}

\usepackage{amsmath}
\usepackage{amsthm}
\usepackage{hyperref}
\usepackage{amsfonts,graphics,amsthm,amsfonts,amscd,latexsym}
\usepackage{epsfig}
\usepackage{flafter}
\usepackage{mathtools}
\usepackage{comment}
\usepackage{stmaryrd}
\usepackage{soul}
\usepackage{enumitem}

\usepackage[abbrev,nobysame,alphabetic]{amsrefs}

\hypersetup{
    colorlinks=true,    
    linkcolor=blue,          
    citecolor=blue,      
    filecolor=blue,      
    urlcolor=blue           
}
\usepackage{tikz}
\usetikzlibrary{graphs,positioning,arrows,shapes.misc,decorations.pathmorphing}

\tikzset{
    >=stealth,
    every picture/.style={thick},
    graphs/every graph/.style={empty nodes},
}

\tikzstyle{vertex}=[
    draw,
    circle,
    fill=black,
    inner sep=1pt,
    minimum width=5pt,
]
\usepackage[position=top]{subfig}
\usepackage{amssymb}
\usepackage{color}

\calclayout

\usetikzlibrary{decorations.pathmorphing}
\tikzstyle{printersafe}=[decoration={snake,amplitude=0pt}]

\newcommand{\Pic}{\operatorname{Pic}}

\newcommand{\rank}{\operatorname{rank}}

\newcommand{\Aut}{\operatorname{Aut}}

\newcommand{\Nef}{\operatorname{Nef}}

\newcommand{\MW}{\operatorname{MW}}
\newcommand{\Bl}{\operatorname{Bl}}

\renewcommand{\setminus}{\backslash}

\usepackage{tikz}
\usetikzlibrary{matrix,arrows,decorations.pathmorphing}

  \newtheorem{theorem}{Theorem}[section]
  \newtheorem{lemma}[theorem]{Lemma}
  \newtheorem{proposition}[theorem]{Proposition}
  \newtheorem{corollary}[theorem]{Corollary}

  \newtheorem{definition}[theorem]{Definition}

  \newtheorem{construction}[theorem]{Construction}

\newtheorem{remark}[theorem]{Remark}

\theoremstyle{remark}

\numberwithin{equation}{section}

\usepackage[all]{xy}

\title{Rational surfaces with a non-arithmetic automorphism group}

\author{Jennifer Li}
\address{
Department of Mathematics\\
Princeton University\\304 Washington Road\\Princeton\\NJ 08544\\USA}
\email{jenniferli@princeton.edu}

\author{Sebasti\'{a}n Torres}
\address{
Departamento de Matem\'atica\\
Universidad T\'ecnica Federico Santa Mar\'ia\\Avenida Espa\~na 1680\\Valpara\'iso\\Chile}
\email{sebastian.torresk@usm.cl}

\subjclass[2020]{14J50, 14J26, 14J27}

\begin{document}

\begin{abstract}
In \cite{T12}, Totaro gave examples of a K3 surface such that its automorphism group is not commensurable with an arithmetic group, answering a question of Mazur \cite[\S7]{M93}. We give examples of rational surfaces with the same property. Our examples $Y$ are Looijenga pairs, i.e., there is a connected singular nodal curve $D \subset Y$ such that $K_{Y} + D = 0$.
\end{abstract}

\maketitle

\section{Introduction}

A Looijenga pair is a pair $(Y, D)$ in which $Y$ is a smooth, complex projective surface and $D\in |-K_Y|$ is a (reduced) anticanonical divisor given by a connected singular nodal curve. Such boundary $D$ must be either a cycle of smooth rational curves or an irreducible rational curve with a single node. Throughout this article we will use some general facts about these surfaces, and more details can be found in \cites{GHK15,F15}. Such a surface $Y$ is necessarily rational (see e.g. \cite[\S1]{GHK15}). We fix an orientation of the cycle $D$ and write $D=D_1+\ldots+D_r$, where $D_i$ are the irreducible components of $D$, and where the order is compatible with the orientation. If $r=2$, then $D$ consists of a cycle of two smooth rational curves intersecting at two points. In the case $r>2$, for every $i\in\mathbb{Z}/r\mathbb{Z}$ the component $D_i$ intersects each of the $D_{i-1}$ and $D_{i+1}$ once, and no other component. We denote by $D_i^\circ$ the component $D_i$ minus the two intersection points with $D_{i\pm1}$.

We say that the pair $(Y,D)$ is negative (semi)definite if the intersection matrix of $D_1,\ldots,D_r$ is so. The sequence $(D_1^2,\ldots,D_r^2)$ is called the self-intersection sequence of $(Y,D)$. If $D_i^2=-2$ for every $i$, then $(Y,D)$ is negative semidefinite. Assuming that there are no $(-1)$-components in the boundary, $(Y,D)$ is negative definite if and only if $D_i^2\leq -2$ for every $i$ and $D_i^2\leq-3$ for at least one index $i$ \cite[\S1]{F15}. The group $\Pic(Y)$ is isomorphic to $H^2(Y,\mathbb{Z})$ via the first Chern class. We denote its rank by $\rho(Y)$, which equals the second Betti number of $Y$. That is, $\rho(Y)=b_2(Y)=10-D^2$, since $K_Y=-D$.

When $(Y,D)$ is negative definite, $D$ can be analytically contracted to a cusp singularity \cite{G62}, and in this way we obtain a normal complex analytic surface $\hat{Y}$ with trivial dualizing sheaf. Thus, $\hat{Y}$ can be regarded as a singular analog of a K3 surface.

\begin{definition}
An automorphism of a Looijenga pair $(Y,D)$ is an automorphism $\varphi:Y \xrightarrow{\sim} Y$ such that $\varphi(D_i)=D_i$ for every $i$ and $\varphi$ preserves the orientation if $r\leq 2$. We denote by $\Aut (Y,D)$ the group of automorphisms of the pair $(Y,D)$.
\end{definition}

\begin{definition}
    Two groups $G$, $H$ are said to be commensurable with each other if they have finite-index subgroups $G'\subset G$ and $H'\subset H$ such that $G'\cong H'$. 
    In this case, we write $G \doteq H$.
\end{definition}

\begin{definition}
An arithmetic group is a subgroup $G$ of the group of $\mathbb{Q}$-points of some $\mathbb{Q}$-algebraic group $H_\mathbb{Q}$ such that, for some integral structure on $H_\mathbb{Q}$, the intersection $G\cap H(\mathbb{Z})$ is of finite index in both $G$ and $H(\mathbb{Z})$.
\end{definition}

In \cite{T12}, Totaro showed that the automorphism group of a K3 surface need not be commensurable with an arithmetic group, thereby answering a question by Mazur \cite[\S 7]{M93}. In \cite{L18}, Lesieutre found a projective variety of dimension six with non-finitely generated (in particular, non-arithmetic) automorphism group. Inspired by Lesieutre's work, Dinh and Oguiso \cite{DO19} found projective varieties of each dimension $d\geq 2$ with non-finitely generated automorphism group. Remarkably, rational projective varieties with the same property have been found in every dimension $d\geq 3$ by Dinh, Oguiso and Yu \cite{DOY22}. In \cite{O20}, Oguiso finds a smooth projective surface in odd characteristic, birational to a K3 surface, with non-finitely generated automorphism group.
In the present article, we find a rational surface, namely a Looijenga pair, with discrete and finitely generated yet non-arithmetic automorphism group. As far as the authors know, this is the first rational example in dimension $2$ with this property.

\begin{theorem}
\label{thm:introduction}
There exists a (negative definite) Looijenga pair $(Y,D)$ such that the automorphism group $\Aut(Y,D)$ is not commensurable with an arithmetic group.   
\end{theorem}

We note that, when $(Y,D)$ is negative definite, $\Aut(Y,D)$ is a finite-index subgroup of $\Aut(Y)$. Indeed, in this case, every automorphism of $Y$ must permute the components $D_1,\ldots,D_r$, which gives a homomorphism $\rho:\Aut(Y)\to S_r$ (where $S_{r}$ denotes the symmetric group on $r$ elements), and $\Aut(Y,D)$ is precisely the subgroup of $\ker\rho$ consisting of automorphisms that preserve the orientation (if $r\leq 2$), which has finite index in $\Aut(Y)$. In particular, both groups are commensurable with each other, and Theorem \ref{thm:introduction} may be considered as a statement about $\Aut(Y)$ itself.

Since every automorphism of $Y$ induces an isometry of the Picard lattice, we have a group homomorphism $\Aut(Y)\to O(\Pic(Y))$, the kernel of which is finite (see e.g. \cite[Proposition 2.1]{DM19}). 
Using this, the proof of Theorem \ref{thm:introduction} reduces to studying groups that live inside $O(\Pic(Y))$.
In order to do this, we follow the techniques used in \cite[\S 7]{T12}. Most importantly, we make use of the following theorem.

\begin{theorem}
\label{thm:Totaro-mainTheorem} \cite[Theorem 7.1]{T12}
Let $M$ be a lattice of signature $(1, m)$ for $m \geq 3$. Let $S$ be an infinite-index subgroup in $O(M)$. Suppose that $\mathbb{Z}^{m-1}$ is an infinite-index subgroup of $S$. Then $S$ is not commensurable with an arithmetic group.
\end{theorem}

We will study $\Aut (Y,D)$ through its action on the perpendicular lattice to $D_1,\ldots,D_r$. In order to construct some infinite subgroups of $\Aut (Y,D)$, we will, as in \cite[\S 7]{T12}, construct some elliptic fibrations and compute their rank using the Shioda--Tate formula.

\subsection{Acknowledgments}
The authors are very grateful to Paul Hacking for bringing their attention to this problem. They also thank Paul Hacking, J\'{anos} Koll\'{a}r, Burt Totaro, and Chenyang Xu for very helpful discussions and comments, as well as the referee for useful feedback. The second author was partially supported by the Ministry of Education and Science of the Republic of Bulgaria through the Scientific Program ``Enhancing the Research Capacity in Mathematical Sciences (PIKOM)''.

\section{Background notions}

Throughout the article, $(Y,D)$ will denote a Looijenga pair with an anticanonical cycle of length $r>1$, together with an orientation of the cycle, as defined in the introduction. 
We need some definitions first.

\begin{definition}
An internal $(-2)$-curve is a smooth rational curve of self-intersection $-2$ disjoint from $D$. The subgroup $W_Y\subset O(\Pic (Y))$ generated by the reflections along internal $(-2)$-curves is called the Weyl group of $Y$. The pair $(Y,D)$ is said to be generic if it has no internal $(-2)$-curves.
\end{definition}

By \cite[Lemma 2.1]{GHK15}, there is a canonical identification $\Pic^0 (D)\cong\mathbb{G}_m$, where $\Pic^0(D)$ is the subgroup of $\Pic (D)$ consisting of line bundles $L$ on $D$ such that $\deg L|_{D_i}=0$ for each $i$. Under this identification, if $p,q\in D_i^\circ$ and we choose coordinates on $D_i^\circ\cong\mathbb{G}_m$ such that $p$ corresponds to $1$ and $q$ corresponds to $\lambda$, then $\mathcal{O}_D(q-p)\in\Pic^0(D)$ corresponds to $\lambda^{-1}\in\mathbb{G}_m$ \cite[Lemma 1.6]{F15}.

Recall that on $\Pic (Y)$ we have an intersection pairing. We use it to define the following sublattice of $\Pic (Y)$.

\begin{definition}
For a Looijenga pair $(Y,D)$, let $\Lambda(Y,D)=\langle D_1,\ldots,D_r\rangle^\perp$ be the sublattice of classes $\alpha\in\Pic (Y)$ such that $\alpha\cdot D_i=0$ for every $i$.
\end{definition}

We have the following Lemma.

\begin{lemma}\cite[Lemma 1.5]{F15}
\label{lem:RankOfM}
Let $s$ be the rank of the kernel of the homomorphism $\bigoplus_{i=1}^r\mathbb{Z}[D_i]\to \Pic (Y)$. Then $\Lambda(Y,D)$ is a lattice of rank $10-D^2-r+s$. In particular, if the classes of $D_1,\ldots,D_r$ are linearly independent in $\Pic (Y)$, then the rank of $\Lambda(Y,D)$ equals $10-D^2-r$.
\end{lemma}

Using the identification $\Pic^0 (D)\cong\mathbb{G}_m$ above, we can define the following.

\begin{definition}
\label{def:periodPoint}
The period point of $(Y,D)$ is the morphism
\begin{center}
$\phi_Y:\Lambda(Y,D)\to\Pic^0(D)\cong\mathbb{G}_m$, given by $L\mapsto L|_D$
\end{center}
\end{definition}

By abuse of notation, we may denote $\phi_Y(\mathcal{O}_Y(E))$ simply by $\phi_Y(E)$, where $E$ is a divisor.
By \cite[\S 2]{GHK15}, any given morphism $\phi:\Lambda(Y,D)\to\Pic^0(D)$ can be realized as the period point of some deformation equivalent pair $(Y',D')$. More precisely, we have the following result:

\begin{theorem} \cite[Theorem 3.17]{F15}
\label{thm:PeriodPointSurjective}
Let $\phi:\Lambda(Y,D)\to\mathbb{G}_m$ be any morphism. Then there is a deformation equivalent pair $(Y',D')$ and an identification $\Lambda(Y,D)\cong \Lambda(Y',D')$ induced by parallel transport, such that the period point $\phi_{Y'}$ of $(Y',D')$ corresponds to $\phi$.
\end{theorem}

We will use this to construct a Looijenga pair whose period point satisfies certain conditions. We may denote the pair $(Y',D')$ from Theorem \ref{thm:PeriodPointSurjective} by $(Y_\phi,D_\phi)$.

\begin{definition}
\label{def:roots}
    We denote by $\Phi$ the set of classes in $\Lambda(Y,D)$ with square $-2$ which are realized by an internal $(-2)$-curve on some deformation equivalent Looijenga pair $(Y',D')$. An element $\alpha\in\Phi$ is called a root.
\end{definition}

\begin{remark}
\label{rem:PeriodPointAndMinus2Curves}
    By \cite[Corollary 3.5]{GHK15}, the condition that $(Y,D)$ is generic is equivalent to having $\phi_Y(\alpha)\neq 1$ for every $\alpha\in\Phi$. In this case, there are no $(-2)$-curves in $Y \setminus D$.
\end{remark}

\section{Elliptic fibrations}

Our goal is to use Theorem \ref{thm:Totaro-mainTheorem}, for which we will need to produce some infinite-order subgroups of the automorphism group of a Looijenga pair. This is achieved by constructing elliptic fibrations with an infinite Mordell--Weil group. In the present section we collect some well-known facts about elliptic fibrations that we will need later.
By an elliptic fibration we will mean a proper morphism $Y\to B$, where $B$ is a smooth curve, such that the general fiber is a smooth curve of genus one.
For our purposes, all elliptic fibrations will be over $B=\mathbb{P}^1$, since under our assumptions $Y$ is a 
smooth projective 
rational surface.

\begin{definition}\label{def:MW}
Let $\pi:Y\to \mathbb{P}^1$ be an elliptic fibration.
We define the Mordell--Weil group of $\pi$ to be
\begin{equation*}
\MW(\pi) = \Pic^{0}(Y_{\eta}),
\end{equation*}
where $Y_{\eta}$ is the generic fiber. 
\end{definition}

The group $\MW(\pi)$ is abelian and finitely generated, and it acts by translation on each smooth fiber of $\pi$. If $\pi: Y \rightarrow \mathbb{P}^1$ is minimal, that is, if there are no $(-1)$-curves in the fibers of $\pi$, then $\MW(\pi)$ acts (faithfully) on $Y$ by automorphisms. If, in addition, $\pi$ has a section, then $\MW(\pi)$ acts simply transitively on the sections of $\pi$ by translation.

\begin{remark}
In Definition \ref{def:MW}, we do not require $\pi$ to have a section. If it does, then $\MW(\pi)$ can be identified with the set of sections.
The latter is sometimes given as a definition in the literature.
\end{remark}

\begin{lemma}
\label{lem:TwoDifferentMW}
    Let $\pi_i:Y\to \mathbb{P}^1$, $i=1,2$, be two different elliptic fibrations, and suppose there are subgroups $G_i\subset\MW(\pi_i)$ such that $G_i$ acts by automorphisms on $Y$. Then the intersection $G_1\cap G_2\subset \Aut(Y)$ is finite.
\end{lemma}

\begin{proof}
Equivalently, we may assume that each $G_i$ is free and show that $G_1\cap G_2$ is trivial.
We follow the proof of \cite[Corollary 7.2]{T12}. By the Hodge Index Theorem, the intersection form has signature $(1,\rho(Y)-1)$ on the Picard group of $Y$. Fix an ample divisor $H$. Then the quotient of the positive cone $C=\{\alpha\in\Pic(Y)\mid \alpha^2>0,\: \alpha\cdot H>0\}$ by $\mathbb{R}_{>0}$ can be identified with the $(\rho(Y)-1)$-dimensional hyperbolic space, and its group of isometries contains $O(\Pic(Y))$ as a discrete subgroup
(cf. \cite[\S5]{T12}).
Let $g\in G_i$ be non-trivial. Using the homomorphism $\Aut(Y)\to O(\Pic(Y))\subset O(C/\mathbb{R}_{>0})$, we identify $g$ with an infinite-order isometry of the hyperbolic space. Then the subgroup generated by $g$ is an infinite discrete subgroup of isometries where each element fixes a point in the boundary of the cone $C$, namely the class of a fiber $F_i$, which has $F_i^2=0$.
Following \cite[Theorem 2.6 of Chapter 5]{AVS93} and its proof, this implies that $g$ must be of parabolic type. Since parabolic motions have only one fixed point in the sphere at infinity, the fact that $F_1\neq F_2$ implies that $G_1\cap G_2$ must be trivial, as desired.
\end{proof}

\begin{lemma}
\label{lem:Period point-elliptic}
Let $D$ be a length $r$ cycle of $(-2)$-curves. Identify $\Pic^0(D)\cong \mathbb{G}_m$ as above and suppose that $\mathcal{O}_Y(D)|_D$ is torsion of order $m$. Then there is a minimal elliptic fibration $\pi:Y\to\mathbb{P}^1$ with $\pi^{*}(\infty)=mD$. In addition, if $m=1$ then $\pi$ has a section.
\end{lemma}

\begin{proof}
    For each $1\leq k\leq m-1$, we twist the tautological short exact sequence by $\mathcal{O}_Y(kD)$ to obtain 
    \begin{equation}\label{eq:ShortSeq}
    0\to \mathcal{O}_Y((k-1)D)\to \mathcal{O}_Y(kD)\to\mathcal{O}_Y(kD)|_D\to 0.
    \end{equation}
First, we see that $H^0(D,\mathcal{O}_D(kD))=H^1(D,\mathcal{O}_D(kD))=0$. Indeed, call $L=\mathcal{O}_Y(kD)|_D$. Then $L$ is torsion in $\Pic^0(D)$ and nontrivial. Using Riemann--Roch,
\begin{equation*}
    \chi(L)=1-p_a(D)+\deg L=\deg L,
\end{equation*}
since $p_a(D)=1$.
But $L$ torsion implies $\deg L=0=h^0(L)-h^1(L)$. Since $L\in \Pic^0(D)$ but $L$ is nontrivial, $h^0(L)=0$ (cf. \cite[Lemma 1.7]{F15}), so we obtain $h^0(L)=h^1(L)=0$, as desired.

As a consequence, from \eqref{eq:ShortSeq} we get an isomorphism $H^{1}(Y,\mathcal{O}_{Y}(k-1)D) \xrightarrow{\sim} H^{1}(Y,\mathcal{O}_{Y}(kD))$ for each $1\leq k\leq m-1$. Since $Y$ is rational, $H^1(Y,\mathcal{O}_Y)=0$, and then by induction we get $H^{1}(Y,\mathcal{O}_{Y}(kD))=0$ for every $1\leq k\leq m-1$. In particular, $H^{1}(Y,\mathcal{O}_{Y}((m-1)D))=0$.

Now consider \eqref{eq:ShortSeq} with $k=m$, that is,
\begin{equation*}
0\to \mathcal{O}_Y((m-1)D)\to \mathcal{O}_Y(mD)\to\mathcal{O}_D\to 0.
\end{equation*}
Since $H^1(Y,\mathcal{O}_Y((m-1)D))=0$, we get a short exact sequence of global sections. Therefore, the constant section $1\in H^0(D,\mathcal{O}_D)$ has a lift $t$ to $H^0(Y,\mathcal{O}_Y(mD))$ that does not vanish anywhere in $D$. On the other hand, the tautological section $s\in H^0(Y,\mathcal{O}_Y(D))$ vanishes precisely along $D$ with multiplicity one. Then $(s^m:t)$ defines a morphism $\pi$ to $\mathbb{P}^1$ with fiber $mD$ over $\infty$.

We see that $\pi$ is an elliptic fibration. Let $F$ be a fiber. Then $\deg K_F= (K_Y+F)\cdot F=-D\cdot (mD)=0$ since $mD$ itself is a fiber. Thus the general fiber has genus one. Furthermore, $\pi$ is minimal. Indeed, any $(-1)$-curve in $Y$ not contained in $D$ must have intersection $1$ with $D=-K_Y$, but a fiber of $\pi$ other than $mD$ does not intersect $D$. 

Finally, when $m=1$, the $(-1)$-curves of $Y$ are precisely the sections. This follows from the fact that a $(-1)$-curve intersects the fiber with $F\cdot C = D\cdot C =1$. Since $Y$ is a rational surface of Picard rank $10-D^{2}=10$, it has a $(-1)$-curve by the classification of surfaces.
\end{proof}

\begin{remark}\label{rem:FiberIs-2}
If $\pi:Y\to \mathbb{P}^1$ is a minimal elliptic fibration, then one of the fibers must be $D$ with some multiplicity. This follows from the Kodaira canonical bundle formula \cite{K64}. 
Moreover, if $F$ is some other reducible fiber, $F$ must be a union of internal $(-2)$-curves with some multiplicities (cf. Kodaira's classification of singular fibers of minimal elliptic fibrations \cite{K63}).
\end{remark}

In order to compute the rank of the Mordell--Weil group of an elliptic fibration, we will use the Shioda--Tate formula \cite[Corollary 1.5]{S72}
(cf. also \cite[Lemma 5.1]{DM22}).

\begin{theorem}[Shioda--Tate formula]
\label{thm:Shioda-Tate}
Let $\pi:Y\to \mathbb{P}^1$ be an elliptic fibration.
Let $\Sigma \subset \mathbb{P}^1$ be the locus where the fibers $\pi^{-1}(p)$ are singular. For each $p\in\Sigma$, let $m_p$ be the number of irreducible components of $\pi^{-1}(p)$. Then $\rho(Y)=\rank\MW(\pi)+2+\sum_{p\in\Sigma}(m_p-1)$.
\end{theorem}

\subsection{The case r=7.}
When $D$ is a cycle of seven $(-2)$-curves, the root system $\Phi$ can be described as follows. Let $(Y_e,D_e)$ be the deformation equivalent pair with a split mixed Hodge structure on $H^{2}(Y_e \setminus D_e)$, or equivalently, such that $\phi_{Y_e}=e$ is the identity (cf. \cite[Proposition 3.12]{F15}). Then the elliptic fibration $Y_e\to\mathbb{P}^1$ given by Lemma \ref{lem:Period point-elliptic} has two reducible fibers, which consist of $D$ and another $I_2$ fiber (a cycle of two $(-2)$-curves). Let $\beta$ and $F-\beta$ denote the two $(-2)$-curves of $I_{2}$, where $F$ is the class of a fiber. Then, by 
\cite[Proposition 9.16 and Corollary 9.20]{F15} 
$\Phi$ is the root system of type $\tilde{A}_{1}$ (that is, affine ${A}_{1}$) associated to the $I_{2}$ fiber, which is given by
\begin{equation}\label{eq:AffineA1}
\Phi=\{\pm\beta+kF \mid k\in\mathbb{Z} \}.
\end{equation}

\begin{lemma}
\label{lem:No-2Curves}
Let $(Y, D)$ be a Looijenga pair where $D$ is a cycle of $(-2)$-curves as above and $\phi_{Y}(D) = 1$. If $\phi_Y(\beta)\neq 1$, then $(Y,D)$ does not contain any internal $(-2)$-curves.
\end{lemma}

\begin{proof}
By Remark \ref{rem:PeriodPointAndMinus2Curves}, we just need to check that $\phi_Y(\alpha)\neq 1$ for every $\alpha\in \Phi$. For a root $\alpha=\pm\beta+kF$, we compute
$$
\phi_Y(\pm\beta+kF)=\phi_Y(\beta)^{\pm1}\phi_Y(F)^k.
$$
But $\phi_Y(F)=\phi_Y(D)=1$, because $D$ is a fiber of the minimal elliptic fibration $Y\to\mathbb{P}^1$, so $\phi_Y(\alpha)=\phi_Y(\beta)^{\pm 1}\neq 1$. 
\end{proof}

Moreover, \eqref{eq:AffineA1} together with Remark \ref{rem:FiberIs-2} implies that if $(Y,D)$ is given by a cycle of seven $(-2)$-curves, then any minimal elliptic fibration $Y\to\mathbb{P}^1$ will have at most one reducible fiber besides $D$, consisting of a union of two $(-2)$-curves.

\begin{corollary}
\label{cor:InfiniteMW}
If $D$ is a cycle of seven $(-2)$-curves and $\pi:Y\to\mathbb{P}^1$ is a minimal elliptic fibration, then the Mordell--Weil group $\MW(\pi)$ is infinite.
\end{corollary}

\begin{proof}
The Picard rank is $\rho(Y)=10-D^2=10$, because $D$ is a cycle of $(-2)$-curves. Using the Shioda--Tate formula (Theorem \ref{thm:Shioda-Tate}), the rank of $\MW(\pi)$ is
$$
\rank\MW(\pi)=10-2-\sum_{p\in\Sigma}(m_p-1).
$$
The reducible fibers are $D$ and at most one extra fiber consisting of two $(-2)$-curves, so
$$
\sum_{p\in\Sigma}(m_p-1)\leq (7-1)+(2-1)=7.
$$
Therefore $\rank\MW(\pi)\geq 10-2-7= 1$ and, in particular, the Mordell--Weil group $\MW(\pi)$ is infinite.
\end{proof}

\section{Main result}

In this section, we construct our main example. Consider the toric Looijenga pair $(\bar{Y}, \bar{D}) = (\bar{Y}_{e}, \bar{D}_{e})$, where $\bar{D}_{e}$ is a cycle of length $r = 7$, with self-intersection cycle $(-1,-2,-1,-1,-1,-1,-2)$, and the period point $\phi_{\bar{Y}}=e$ is the identity.
The toric variety $\bar{Y}$ is determined by a fan with seven rays in $\mathbb{Z}^2$, so in particular $\rho(\bar{Y}) = 7 - 2 = 5$. Blowing up one interior point on each of the $(-1)$-components of $\bar{D}$ yields a pair $(Y,D)$ with a cycle of seven $(-2)$-curves 
(see Figure \ref{fig:YDn7-noTorsion}) 
and $\rho(Y)=10$. We construct a pair $(Y,D)$ by choosing the blow-up points in the following way.

\begin{figure}
\centering
\includegraphics[scale=0.22]{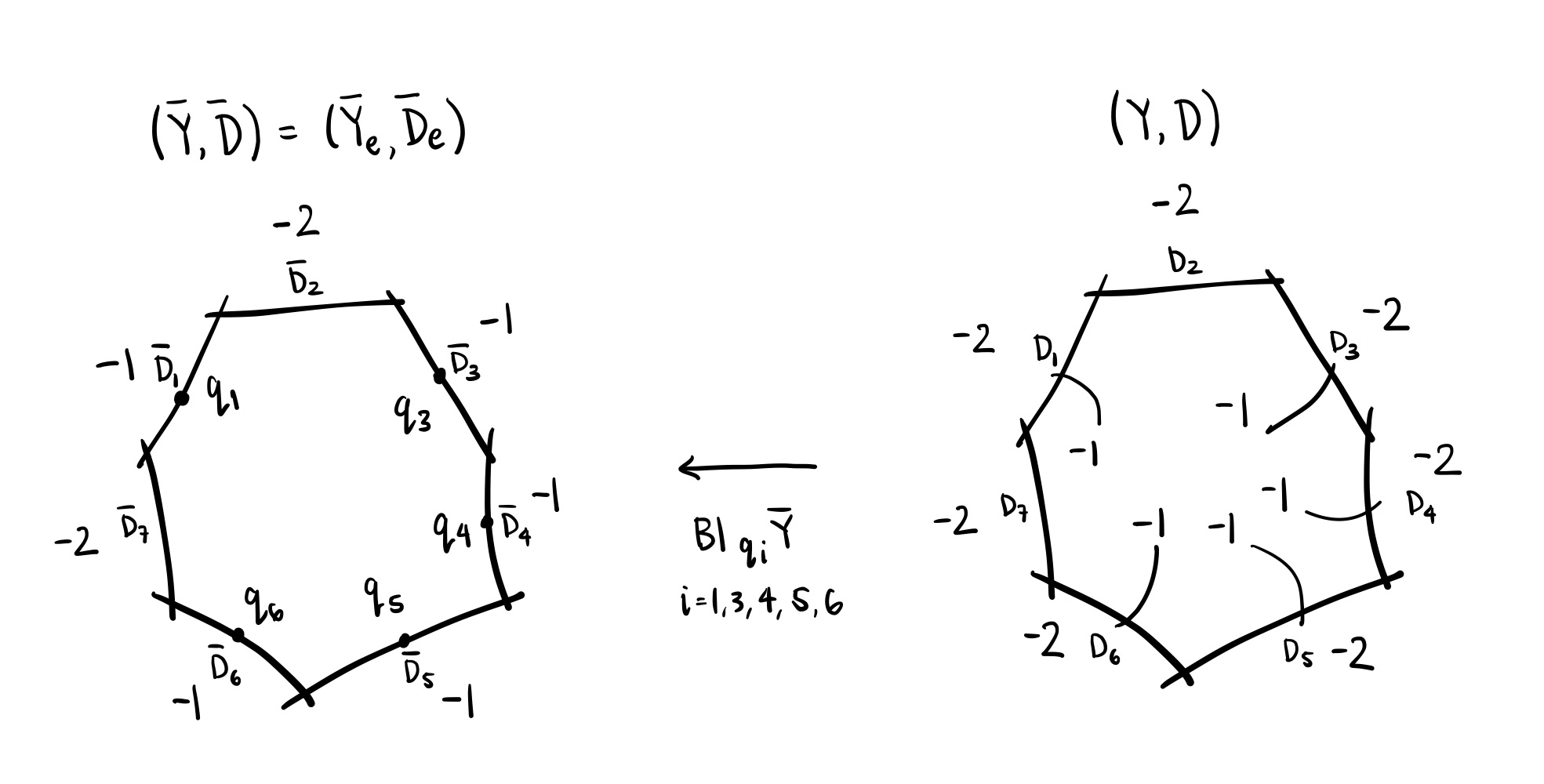}
\caption{$Y=\Bl_{q_i}\bar{Y}$}
\label{fig:YDn7-noTorsion}
\end{figure}

\begin{construction}
\label{cons:(Y,D)}
Let $(Y,D)$ be the blow-up of $(\bar{Y},\bar{D})$ in one point $q_i\in \bar{D}_i^\circ$ for each $(-1)$-component $\bar{D}_i$, in such a way that $\phi_Y:\Lambda(Y,D)\to\Pic^0(D)\cong\mathbb{G}_m$ is torsion, $\phi_Y(D)=1$ and $\phi_Y(\beta)\neq 1$, where $\beta$ is the class appearing in \eqref{eq:AffineA1}.
\end{construction}

The existence of $(Y,D)$ with such choice for $\phi_Y$ is guaranteed by Theorem \ref{thm:PeriodPointSurjective}.

\begin{remark}
\label{rem:familyOfExamples}
There is some flexibility, in that $\phi_Y$ can be chosen to be any morphism satisfying the conditions stated in Construction \ref{cons:(Y,D)}. Therefore, our main theorem is not strictly a statement about a single Looijenga pair, but rather a collection of examples.
\end{remark}

\begin{proposition}
\label{prop:MinimalElliptic}
Let $(Y,D)$ be as in Construction \ref{cons:(Y,D)}. Then there is a minimal elliptic fibration $\pi_1:Y\to \mathbb{P}^1$ with $\pi_1^{*}(\infty)=D$ and no other reducible fibers, and admitting a section.
\end{proposition}

\begin{proof}
The existence of the minimal elliptic fibration with $\pi_1^{*}(\infty)=D$ and admitting a section is guaranteed by Lemma \ref{lem:Period point-elliptic} and the fact that $\phi_Y(D)=1$, which by definition means that $\mathcal{O}_Y(D)|_D$ is trivial. Any other reducible fiber would consist of a union of internal $(-2)$-curves, with some multiplicities. But by Lemma \ref{lem:No-2Curves}, $(Y,D)$ has no internal $(-2)$-curves.
\end{proof}

Since $\pi_1$ has a section, $\MW(\pi_1)$ acts simply transitively on the sections of $\pi_1$ by translation. 
Now let $C$ and $C'$ be two sections of $\pi_1$ and suppose that $C$ and $C'$ intersect the same component $D_i$, that is, $C|_D=\mathcal{O}_D(p)$ and $C'|_D=\mathcal{O}_{D}(q)$ for some $p,q\in D_i^\circ$. Then $C-C'\in \Lambda(Y,D)$ and therefore $(C-C')|_D=\mathcal{O}_D(p-q)$ must be torsion of order at most the order of $\phi_Y$. Hence, there are only finitely many choices for such $\mathcal{O}_D(p-q)$. But there are infinitely many sections, as the rank of $\MW(\pi_1)$ is positive by Corollary \ref{cor:InfiniteMW}. Then there must be a component $D_i$ and some point $p\in D_i^\circ$ through which there pass infinitely many such sections of $\pi_1$, and each section is a $(-1)$-curve. Thus, we have proved the following proposition.

\begin{proposition}
\label{prop:SectionsThru1pt}
Let $(Y,D)$ be as in Construction \ref{cons:(Y,D)} and $\pi_1$ as in Proposition \ref{prop:MinimalElliptic}. There is a component $D_i$ of $D$ and a point $p\in D_i^\circ$ through which there pass infinitely many $(-1)$-curves.
\end{proposition}

\begin{construction}
\label{cons:(Y,D)tilde}
Let $(Y,D)$ be as in Construction \ref{cons:(Y,D)} and $\pi_1$ as in Proposition \ref{prop:MinimalElliptic}. Let $p$ be a point as in Proposition \ref{prop:SectionsThru1pt}, through which there pass infinitely many sections of $\pi_1$. Let $(\widetilde{Y},\widetilde{D})$ be the Looijenga pair where $\widetilde{Y}$ is obtained by blowing up $Y$ at the point $p$ and $\widetilde{D}$ is the strict transform of $D$.
\end{construction}

Observe that, for each $(-1)$-curve in $Y$ passing through $p$, its strict transform becomes a $(-2)$-curve in $\widetilde{Y}$. Moreover, it is disjoint from $\widetilde{D}$. In particular, $\widetilde{Y}$ contains infinitely many internal $(-2)$-curves, and therefore its Weyl group $W_{\widetilde{Y}}$ is infinite. This fact will be crucial for the proof of Theorem \ref{thm:MainThm}.
We also note that $\widetilde{D}$ is a cycle of six $(-2)$-curves and one $(-3)$-curve, which implies that $(\widetilde{Y},\widetilde{D})$ is negative definite and $\widetilde{D}^2=-1$. 

\begin{remark}
Since $\phi_{\widetilde{Y}}$ is torsion, the automorphism group of $\widetilde{Y}$ must be finitely generated by \cite{L22}.
\end{remark}

\begin{theorem}
\label{thm:MainThm}
Let $(\widetilde{Y},\widetilde{D})$ be the Looijenga pair from Construction \ref{cons:(Y,D)tilde}. Then $\Aut(\widetilde{Y},\widetilde{D})$ is not commensurable with an arithmetic group.
\end{theorem}

\begin{proof}
We use Theorem \ref{thm:Totaro-mainTheorem} as follows. Let $M=\Lambda(\widetilde{Y},\widetilde{D})$. Since $\widetilde{D}_1,\ldots,\widetilde{D}_7$ are linearly independent, by Lemma \ref{lem:RankOfM}, $M$ is a lattice of rank $10+1-7=4$. Moreover, by the Hodge Index Theorem, the signature of $M$ is $(1,3)$. Every automorphism $\varphi\in\Aut(\widetilde{Y},\widetilde{D})$ preserves the lattice $M$, so we get a group homomorphism $\Aut(\widetilde{Y},\widetilde{D})\to O(M)$, whose image we denote by $S$. Then $S$ is isomorphic to the quotient of $\Aut(\widetilde{Y},\widetilde{D})$ by a finite subgroup.
Indeed, if an automorphism $\varphi$ acts trivially on $M$ in addition to fixing each $\widetilde{D}_i$, then $\varphi$ must act trivially on $\Pic(\widetilde{Y})$. But the subgroup of $\Aut(\widetilde{Y})$ that acts trivially on $\Pic(\widetilde{Y})$ is finite.
As a consequence, we obtain that $S\doteq \Aut(\widetilde{Y},\widetilde{D})$, since the latter is virtually torsion-free.
In particular, it suffices to show that $S$ is not commensurable with an arithmetic group.

First we need to show that $S$ has infinite index in $O(M)$. For this, consider the Weyl group $W_{\widetilde{Y}}\subset O (\Pic(\widetilde{Y}))$, which is of infinite order as seen earlier. It is not difficult to see that, given a class $\alpha$ of an internal $(-2)$-curve, the simple reflection $s_\alpha$ fixes the class of each $\widetilde{D}_i$. Therefore, $W_{\widetilde{Y}}$ acts faithfully on $M$, and by abuse of notation we write $W_{\widetilde{Y}}\subset O(M)$. We will see that $W_{\widetilde{Y}}$ and $S$ intersect trivially.

Let $C^{++}_{D}\subset \Pic(\widetilde{Y})$ be the full-dimensional subcone defined in \cite[Definition 1.7]{GHK15}. 
By \cite[Theorem 3.2]{GHK15}, the set of internal $(-2)$-curves determines a decomposition of $C^{++}_D$ into a collection of Weyl chambers on which $W_{\widetilde{Y}}$ acts simply transitively. One of these chambers is identified with the ample cone of $\widetilde{Y}$, in the sense that its closure coincides with $\Nef(\widetilde{Y})$. Moreover, $M$ intersects the closure of $C^{++}_{D}$ in a face $F$. Hence, intersecting with $M$ we get a chamber decomposition of $F$, where $W_{\widetilde{Y}}$ acts simply transitively on the set of chambers, whereas $\Aut(\widetilde{Y},\widetilde{D})$ (and therefore also $S$) preserves the chamber $F\cap \Nef(\widetilde{Y})$. We conclude that $W_{\widetilde{Y}}\cap S$ is trivial, so $S$ has infinite index in $O(M)$.

Now consider the minimal elliptic fibration $\pi_1:Y\to\mathbb{P}^1$ given by Proposition \ref{prop:MinimalElliptic} and let $\widetilde{\pi}_1:\widetilde{Y}\to\mathbb{P}^1$ be its composition with the blow-up $\widetilde{Y}=\Bl_pY\to Y$. The only reducible fiber of $\pi_{1}$ is $D$. Using the Shioda--Tate formula, we find that $\rank\MW(\pi_{1})=10-2-(7-1)=2$.
The group $\MW(\pi_1)=\MW(\widetilde{\pi}_1)$ acts on $Y$ by automorphisms, and we let $G$ be (the free part of) the subgroup of $\MW(\pi_1)$ consisting of automorphisms that
fix the point $p$. Then $G$ has finite index in $\MW(\pi_1)$ because $\phi_Y$ is torsion (cf. the argument for Proposition \ref{prop:SectionsThru1pt}), and therefore $G\cong\mathbb{Z}^2$. Since the automorphisms in $G$ lift to $\widetilde{Y}$, we have $G\subset\Aut(\widetilde{Y})$, and since $(\widetilde{Y},\widetilde{D})$ is negative definite, passing to a finite-index subgroup we can assume $G\subset\Aut(\widetilde{Y},\widetilde{D})$. Recall that the homomorphism $\Aut(\widetilde{Y},\widetilde{D})\to S\subset O(M)$ has finite kernel, so it sends $G$ into an isomorphic copy $G\cong\mathbb{Z}^2\subset S$.

We claim that $G$ has infinite index in $S$. To see this, we construct a second elliptic fibration as follows. Recall the Construction \ref{cons:(Y,D)tilde} of $(\widetilde{Y},\widetilde{D})$ from $(Y,D)$. Let $q\in D_i^\circ$ be another point in the same component as $p$, such that $\mathcal{O}_Y(q-p)$ is a primitive $m$-th root of unity $\xi_m$ where $m\neq 1$, and such that through $q$ there passes another section $C:=\sigma(\mathbb{P}^1)$ of $\pi_1$. Such point $q$ exists because we have chosen $\phi_Y$ to be torsion and not the identity. Then $C$ is a $(-1)$-curve in $Y$, and its strict transform $\widetilde{C}\subset \widetilde{Y}$ is a $(-1)$-curve in $\widetilde{Y}$. Let $(Y',D')$ be the Looijenga pair obtained by blowing down $\widetilde{C}$. Then $D'$ is a cycle of seven $(-2)$-curves, and by construction $\phi_{Y'}(D')=\xi_m$. By Lemma \ref{lem:Period point-elliptic}, there is a minimal elliptic fibration $\pi_2:Y'\to\mathbb{P}^1$ with $mD$ as a non-reduced fiber. Let $\widetilde{\pi}_2:\widetilde{Y}\to\mathbb{P}^1$ be its composition with the blow-up $\widetilde{Y}=\Bl_qY'\to Y'$. By Corollary \ref{cor:InfiniteMW}, $\MW(\pi_2)$ is infinite. Similarly as before, there is a finite-index subgroup $H\subset \MW(\pi_{2})=\MW(\widetilde{\pi}_{2})$ such that $H\subset \Aut(\widetilde{Y},\widetilde{D})$, and 
by Lemma \ref{lem:TwoDifferentMW}
we see that $G\cap H$ is finite 
(and in fact it must be trivial since $G$ has been chosen to be torsion-free).
This proves that $G$ has infinite index as a subgroup of $S$. By Theorem \ref{thm:Totaro-mainTheorem}, we conclude that $S$ is not commensurable with an arithmetic group, and then the same is true for $\Aut(\widetilde{Y},\widetilde{D})$.
\end{proof}


\end{document}